\documentclass[11pt, a4paper]{amsart}
\usepackage{amsfonts,amssymb,amsmath,amsthm,amscd,mathtools,multicol,tikz, tikz-cd,caption,enumerate,mathrsfs,thmtools,cite}
\usepackage{inputenc}
\usepackage[foot]{amsaddr}
\usepackage[pagebackref=true, colorlinks, linkcolor=blue, citecolor=red]{hyperref}
\usepackage{latexsym}
\usepackage{fullpage}
\usepackage{microtype}
\usepackage{subfiles} 
\usepackage{palatino}
\makeatletter
\makeindex 
\newcommand{\be}{\begin{equation}}
\newcommand{\ee}{\end{equation}}
\newcommand{\beano}{\begin{eqn*}} 
	\newcommand{\eeano}{\end{eqnarray*}}
\newcommand{\ba}{\begin{array}}
	\newcommand{\ea}{\end{array}}

\declaretheoremstyle[headfont=\normalfont]{normalhead}

\newtheorem{theorem}{Theorem}[section]

\newtheorem{lemma}[theorem]{Lemma}

\newtheorem{definition}[theorem]{Definition}
\newtheorem{proposition}[theorem]{Proposition}
\newtheorem{remark}[theorem]{Remark}

\numberwithin{equation}{section}

\newtheorem{thm}{Theorem}[section]

\newtheorem{ex}[thm]{Example}

\begin{document}
\title{Decomposition of quandle rings of Takasaki quandles}
\author{Dilpreet Kaur}
\email{dilpreetkaur@iitj.ac.in}
\address{Indian Institute of Technology Jodhpur}

\author{Pushpendra Singh}
\email{singh.105@iitj.ac.in}
\address{Indian Institute of Technology Jodhpur}

\thanks{We are thankful to Mohamed Elhamdadi for his useful suggestions.The first named author would like to acknowledge the support of SERB through MATRICS project MTR/2022/000231.}
\maketitle

\begin{abstract}
Let $K =  \mathbb{R}$ or $\mathbb{C}$ and $T_{n}$ denote the Takasaki quandle of order $n$. In this article we provide decomposition of quandle ring $K[T_n]$ as direct sum of right simple ideals. This decomposition is equivalent to decomposition of regular representation \cite{EM18} of Takasaki quandles.
\end{abstract}

\section{Introduction}
A quandle is an algebraic system of a set together with a binary operation that satisfies three axioms derived from the Reidemeister moves on oriented link diagrams. The concept of quandles was introduced independently by Joyce \cite{DEJ} and Matveev \cite{SVM} for constructing invariants of knots and links. However, this concept can be traced back to the work of Mituhisa Takasaki, who named it Kei \cite{TM43} which is a particular type of quandle known as involutary quandle.

Quandles are studied extensively from the topological point of view with focus on their study as invariants in knot theory. However, due to their algebraic structure, recently, mathematicians are exploring quandles from algebraic point of view. In this regard, the concept of quandle rings is introduced in \cite{BPS19}, analogous to group rings. It has been noted that contrary to group rings, quandle rings are non-associative (except when the quandle is trivial). In \cite{BPS19} and \cite{BPS22}, authors discussed various properties of quandle rings, such as the relation between subquandles of quandles and ideals of quandle rings, power associativity of quandle rings, existence of zero divisors, computation of idempotent elements of integral quandle rings, commutator sub algebras etc.
 
  In the paper \cite{BPS19}, authors also posed various open problems, such as the isomorphism problem for quandle rings. In this regard, authors of \cite{EFT} gave an example of isomorphic quandle rings for non isomorphic quandles for rings of certain characteristics. They also gave results for the relation between quandles for their quandle rings to be isomorphic under certain conditions. Furthermore, the concept of derivation algebra of quandle algebra is studied in \cite{EMSZ22} where a complete description of derivation algebra for quandle algebra of dihedral quandles is provided. The notion of representations of quandles is introduced in \cite{EM18} analogous to the representations of groups.
  
   In this article, we study the decomposition of quandle rings. This concept is studied by authors of \cite{EFT} and \cite{KS22}, where they gave the decomposition of quandle rings of dihedral quandles. In this article we generalize this result to a broader class of quandles known as Takasaki quandles. The article is organized as follows, in section \ref{s2} we recall definitions and examples of quandles and quandle rings. We also provide a brief explanation of decomposition of quandle rings as direct sum of right simple ideals.  In section \ref{s3} we give detailed description of inner automorphism group and orbit structures of Takasaki quandles. In section \ref{Cl} and section \ref{reps} we discuss conjugacy classes and representation theory of the obtained inner automorphism groups.  In section \ref{s6} we provide a complete description of the decomposition of quandle rings of Takasaki quandles.

\section{Preliminaries for Quandles and  Quandle rings} \label{s2}

\begin{definition}
A set $X$ along with a binary operation $\triangleright$ is called quandle if it satisfies the following axioms:
\begin{itemize}
\item[•]
$x \triangleright x =x $ for all $x \in X$.
\item[•]
For any pair $x,y \in X$, there exist a unique $z \in X$ such that $x=z \triangleright y$.
\item[•]
$(x \triangleright y)\triangleright z = (x \triangleright z ) \triangleright (y \triangleright z) $ for all $x,y,z \in X$.
\end{itemize}
\end{definition}

If first axiom is omitted then structure $(X, \triangleright)$ is called \textit{rack}. We note that last two axioms can be combined and stated as the map $R_{x}:X \to X $ defined as $R_{x}(y)=y \triangleright x$ is an automorphism. These maps $R_{x}$ are called right multiplication maps or symmetries. Similarly left multiplication maps are defined as $L_{x}(y) = x\triangleright y$. Note that $L_{x}$ need not to be a bijective map. Now we give some examples of quandles:

\begin{itemize}
\item
Let $X = \{x_0,x_1,\hdots, x_{n-1}\}$ be a set, $X$ forms quandle with $x_{i} \triangleright x_{j} = x_{i}$ for all $x_i,x_j \in X$. It is known as trivial quandle of order $n$.
\item
Any group $G$ with the operation $g\triangleright h=hgh^{-1}$ becomes a quandle. This is called conjugation quandle and is denoted by $Conj(G)$.
\item
Group $G$ forms a quandle with respect to operation $g \triangleright h = hg^{-1}h$. In particular, if $G$ is abelian group then this quandle is called Takasaki quandle.
\item
The group $\mathbb{Z}_{n}=\{0,1,2,\hdots n-1\}$ together with operation $a \triangleright b= 2b-a~ (mod~n)$ is known as dihedral quandle. It is denoted by $\mathcal{R}_{n}$. 
\item
Let $G$ be additive abelian group, let $\phi \in Aut(G)$.
The group $G$ forms a quandle with the binary operation $a \triangleright b=\phi(a)+(id_{G}-\phi)b$. It is called Alexander quandle.

\end{itemize}

For every quandle $X$, the set of all symmetries generate a group under the composition operation. This group is called inner automorphism group of $X$ and it is denoted by $Inn(X)$. The group $Inn(X)$ acts on quandle $X$ naturally, thus dividing it into union of disjoint orbits.

A quandle $X$ is called connected if the action of $Inn(X)$ on $X$ is transitive. If all left multiplication maps $L_x$ defined as above are bijective then quandle is called latin quandle. We note that every latin quandle is a connected quandle. However there exist connected quandles which are not latin.

In \cite{BPS19}, authors introduced the concept of quandle rings analogous to group rings. Let $K$ be a ring and $(X,\triangleright)$ be a quandle. The quandle ring $K[X]$ is the set of finite linear combinations of quandle elements given as $\left\{ \sum_{i}\alpha_{i}x_{i}~|~\alpha_{i} \in K, x_{i} \in X  \right\}$. It forms an abelian group under point wise addition. The following multiplication equips it with the structure of a ring:
\begin{align*}
\left( \sum_{i}\alpha_{i}x_{i}  \right)\cdot \left( \sum_{j}\beta_{j}x_{j} \right)= \sum_{i,j}\alpha_{i}\beta_{j} \left( x_{i} \triangleright x_{j} \right)
\end{align*}

The quandle ring $K[X]$ is not associative ring in general as quandle operation is not associative except when $X$ is trivial quandle. Moreover, $K[X]$ does not have multiplicative identity.

In \cite{EFT}, authors discussed the decomposition of $K[X]$ into right simple ideals. We define $\rho : {\rm Inn}(X) \to {\rm GL}(K[X])$ by $\rho(R_x)(y) = y\triangleright x$ for all $R_x\in {\rm Inn}(X)$ and $y\in X.$ The map $\rho$ is a representation of ${\rm Inn}(X)$ and decomposing $K[X]$ into right simple ideals is equivalent to decomposing $\rho$ into irreducible representations of the group ${\rm Inn}(X).$

Let $X$ be a quandle of order $n$ and $\mathcal S_n$ be the symmetric group of degree $n$. Suppose $X$ is not a connected quandle and $X=X_1\cup X_2 \cup \dots X_k,$ where $X_i$ is an orbit for action of ${\rm Inn}(X)$ on $X$. We recall from \cite{EFT} that the map $\Phi : X\to \mathcal S_n$ defined by $\Phi(x)=R_x$ restricts to the maps $\Phi_i : X \to \mathcal S_{n_i},$ where $n_i=|X_i|$ for all $1\leq i\leq k$. We denote $Inn(X_i)$ as group generated by $\Phi_{i}(X).$ In this case, $K[X]=\oplus_{i=1}^{k} K[X_i]$ and decomposition of $K[X_i]$ into right simple ideals is equivalent to decomposing it as direct sum of irreducible representations of the group ${\rm Inn}(X_i).$ From now on by module we mean right module and fix ground field to $K = \mathbb{R}$ or $\mathbb{C}$.

\section{Inner Automorphism Groups and Orbits of Takasaki Quandles} \label{s3}

Let $H$ be a finite additive abelian group. Then $H$ together with binary operation $a \triangleright b = 2b-a$ forms a quandle known as Takasaki quandle. We denote it by $(T,\triangleright)$.
More specifically by fundamental theorem of finitely generated abelian groups, the group $H$ is isomorphic to direct product of cyclic groups say $ \mathbb{Z}_{n_{1}} \times \mathbb{Z}_{n_{2}} \times \hdots \times \mathbb{Z}_{n_{k}} $. From now onwards, in this section, $H$ denotes the group isomorphic to $\mathbb{Z}_{n_{1}} \times \mathbb{Z}_{n_{2}} \times \hdots \times \mathbb{Z}_{n_{k}}$. The quandle operation is defined as
$$(i_1,i_2,\hdots,i_k) \triangleright (j_1,j_2,\hdots, j_k) = (2j_1-i_1~(mod~n_1),2j_2-i_2~(mod~n_2),\hdots, 2j_k-i_k~(mod~n_k))$$ 
It forms quandle $(T,\triangleright)$ with above operation.
We set the following notation : $e_{0} = (0,0,\hdots,0)$ and $e_{i} = (0,0,\hdots,1,\hdots,0)$  where only $i'th$ co-ordinate is $1$ and all other co-ordinates are zero, as elements of group $H$.

\begin{lemma}\label{transpositions}
 Let $(T,\triangleright)$ be Takasaki quandle for abelian group $H$. Then the right multiplication map $R_{j}$ is product of transpositions for all $j \in T$.
\end{lemma}

\begin{proof}
$R_{j}(i) = i \triangleright j = 2j-i~$ and $R_{j}(2j-i) = (2j-i) \triangleright j = i$ for all $i \in T$. Thus $R_{j}$ is product of transpositions.
\end{proof}
The following result regarding inner automorphism group for Takasaki quandle is also mentioned for Takasaki quandles of odd order in \cite{BDS17}. However we include this for sake of completeness. We begin with the following Lemma.

\begin{lemma}\label{Order}
 Let $(T,\triangleright)$ be Takasaki quandle for abelian group $H$. Then $R_{e_{i}}R_{e_{0}}$ has order $n_i$, if $n_i$ is odd and $n_i/2$, if $n_i$ is even.
\end{lemma}

\begin{proof}
Let $n_i$ be odd. We get that
$$R_{e_{i}}R_{e_{0}}(j_1,j_2,\hdots,j_k) = R_{e_{i}}(-j_1,-j_2,\hdots,-j_k)$$
$$=(j_1,j_2,\hdots,2+j_i,\hdots,j_k)$$
Now upon composing this $t$ times, we get 
$$(R_{e_{i}}R_{e_{0}})^{t}(j_1,j_2,\hdots,j_k) = (j_1,j_2,\hdots,2t+j_i,\hdots,j_k)$$
If $n_i$ is odd then $2t \equiv 0 ~(mod~n_i)$ only when $t=n_i$. Thus order of $R_{e_{i}}R_{e_{0}}$ is $n_i$. If $n_i$ is even, then order of $R_{e_{i}}R_{e_{0}}$ is equal to $n_i/2$.
\end{proof}

\begin{definition} Let $H$ be a finite abelian group. Then the group $H \rtimes \mathbb{Z}_{2}$ with following presentation
$$\{H,s~|~s^2,shs^{-1} = h^{-1},~ h \in H\}$$ is called generalised dihedral group. We denote this group with $Dih(H)$.
\end{definition}

\begin{lemma}\label{InnerAutomorphismGroup}
Let $(T,\triangleright)$ be Takasaki quandle for abelian group $G = \mathbb{Z}_{n_{1}} \times \mathbb{Z}_{n_{2}} \times \hdots \times \mathbb{Z}_{n_{k}} \times \mathbb{Z}_{n_{k+1}} \times \hdots \times \mathbb{Z}_{n_{k+d}}$ with $gcd(n_i,2) = 2, ~\forall~ i=1,2,\hdots k$ and $gcd(n_i,2) = 1, ~\forall~i= k+1,k+2,\hdots ,k+d$ Then $Inn(T) \cong 2G \rtimes \mathbb{Z}_{2}$, where $2G = \{g+g~|~g \in G\}$ is a subgroup of $G$.
\end{lemma}

\begin{proof}
From Lemma \ref{transpositions} and Lemma \ref{Order}, it follows  that $|R_{e_{0}}| = 2$ and $|R_{e_{i}}R_{e_0}| = {n_i}/2$ for all $i=1,2,\hdots,k$ and $|R_{e_{i}R_{e_{0}}}| = n_{i}$ for all $i= k+1,\hdots,k+d$. Let $\tilde{G}$ be the group generated by the set $\{R_{e_i}R_{e_{0}};~1\leq i \leq k+d\}$. Observe that $R_{e_{i}}R_{e_{0}}R_{e_{j}}R_{e_{0}} =R_{e_{j}}R_{e_{0}}R_{e_{i}}R_{e_{0}} $ for all $i,j=1,2,\hdots,k+d$. Thus $\tilde{G}$ is abelian group and $\tilde{G} \cong 2G \cong \mathbb{Z}_{n_{1}/2} \times \mathbb{Z}_{n_{2}/2} \times \hdots \times \mathbb{Z}_{n_{k}/2} \times \mathbb{Z}_{n_{k+1}} \times \hdots \times \mathbb{Z}_{n_{k+d}}$. Now we prove that $R_{e_{0}}hR_{e_{0}} = h^{-1}$ for all $h \in \tilde{G}$. It is enough to prove this relation for generators of $\tilde{G}$. Again by using Lemma \ref{transpositions}, we get $R_{e_{0}}(R_{e_{i}}R_{e_{0}}){R_{e_{0}}}^{-1} = (R_{e_{i}}R_{e_{0}})^{-1}$ for all $1\leq i \leq k+d$. This proves our result.
\end{proof}

The following Lemma gives a useful presentation of right multiplication map $R_j$ of Takasaki quandle $(T,\triangleright)$ for abelian group $H$.
\begin{lemma}\label{Remark1}
An arbitrary right multiplication map $R_{(j_1,j_2,\hdots,j_k)}$ can be written in terms of elements of $Inn(T)$ as $$R_{(j_1,j_2,\hdots,j_k)} = \left( \prod_{i=1}^{k}(R_{e_{i}}R_{e_{0}})^{j_{i}} \right )R_{e_{0}} $$
\end{lemma}

\begin{proof}
Let $\left(a_1,a_2,\hdots,a_k \right) \in H$, by direct calculations, we get
$$\left( \prod_{i=1}^{k}(R_{e_{i}}R_{e_{0}})^{j_{i}} \right )R_{e_{0}}(a_1,a_2,\hdots,a_k)=\left( \prod_{i=1}^{k}(R_{e_{i}}R_{e_{0}})^{j_{i}} \right )(-a_1,-a_2,\hdots,-a_k)$$

Now lets look at $k'th $ coordinate, we get
\begin{align*}
(R_{e_{k}}R_{e_{0}})^{j_{k}}(-a_1,-a_2,\hdots,-a_k)&=(-a_1,-a_2,\hdots,2j_k-a_k)
\end{align*}

Similarly each $(R_{e_{i}}R_{e_{0}})^{j_{i}}$ only changes $i'th$ coordinate and so we get
\begin{align*}
\left( \prod_{i=1}^{k}(R_{e_{i}}R_{e_{0}})^{j_{i}} \right )R_{e_{0}}(a_1,a_2,\hdots,a_k)&=(2j_1-a_1,\hdots,2j_k-a_k)\\
&=R_{\left( j_1,j_2,\hdots,j_k \right)}\left( a_1,a_2,\hdots,a_k \right)
\end{align*}

This proves the result. 
\end{proof}
Now we provide definition of direct product of quandles:

\begin{definition}
Let $(X,\triangleright)$ and $(Y,\ast)$ be two quandles then the set $X \times Y$ together with operation $(x_1,y_1) \cdot (x_2,y_2) = (x_1 \triangleright x_2,y_1 \ast y_2) $ forms quandle. This is called direct product of two quandles.

\end{definition}

Let $\mathbb{Z}_{m}$ and $\mathbb{Z}_{n}$ be cyclic groups and let $\mathcal{R}_{m}$ and $\mathcal{R}_{n}$ denote the corresponding dihedral quandles. Then we observe that $\mathcal{R}_{m} \times \mathcal{R}_{n}$ is Takasaki quandle for abelian group $\mathbb{Z}_{m} \times \mathbb{Z}_{n}$. Thus we can say that Takasaki quandles are direct product of dihedral quandles. Recall that group $Inn(X)$ acts of quandle $(X,\triangleright)$ as $R_{j}(i) = i \triangleright j$. Now we discuss the orbits of quandle $(T, \triangleright)$ under the action of $Inn(T)$. We recall some results from \cite{SK} with slight change in notations.
\begin{lemma} \cite[Lemma 3.1]{SK} \label{Lemma 3.1SK}
Let $\{Q_i\}_{i=1}^{n}$ be a family of quandles and $Q = \prod_{i=1}^{n}Q_i$. Then $X \subset Q$ is an orbit of $Q$ if and only if $X = \prod_{i=1}^{n}X_i$, where $X_i$ is an orbit of $Q_i$ for all $1\leq i \leq n$.

\end{lemma}

\begin{lemma}\cite[Proposition 3.2]{SK} \label{Proposition 3.2SK}
The Takasaki quandle of $ G = \mathbb{Z}_{n_{1}} \times \mathbb{Z}_{n_{2}} \times \hdots \times \mathbb{Z}_{n_{k}} \times \mathbb{Z}_{n_{k+1}} \times \hdots \times \mathbb{Z}_{n_{k+d}}$ with $gcd(n_i,2) = 2, ~\forall~ i=1,2,\hdots k$ and $gcd(n_i,2) = 1, ~\forall~i= k+1,k+2,\hdots ,k+d$ has $2^k$ orbits.

\end{lemma}

Let $(T,\triangleright)$ be Takasaki quandle as defined in Lemma \ref{Proposition 3.2SK}. From Lemma \ref{Lemma 3.1SK}, we get that the orbits of Takasaki quandle $(T,\triangleright)$ are exactly the product of orbits of individual components of $G$ as quandles.  Each even component $\mathbb{Z}_{n_{i}}$ for $1\leq i \leq k$ has two orbits, one consisting of even numbers and other consisting of odd numbers. We call them even orbit and odd orbit respectively. Each odd component $\mathbb{Z}_{n_{i}}$ for $ k+1 \leq i \leq k+d $ has only one orbit. We call them transitive orbits. Hence by Lemma \ref{Lemma 3.1SK}, we get $2^k$ orbits of action of $Inn(T)$ on $\left( T, \triangleright \right)$. We discuss one example for the sake of clarity of notations.

\begin{ex}
Let $(T,\triangleright)$ be Takasaki quandle for $G = \mathbb{Z}_{4} \times \mathbb{Z}_{6} \times \mathbb{Z}_{3}$. We know that each cyclic component of $G$ corresponds to a dihedral quandle. So $\mathcal{R}_{4}$ has two orbits $\{0,2\}$ and $\{1,3\}$. Similarly $\mathcal{R}_{6}$ has two orbits $\{0,2,4\}$ and $\{1,3,5\}$ and $\mathcal{R}_{3}$ has one orbit $\{0,1,2\}$. There are four ways to do the direct product thus $(T,\triangleright)$ has four orbits. We denote even orbits, odd orbits, transitive orbits with symbol $0,1,0$ respectively.
Thus four orbits are
\begin{align*}
X_{(0,0,0)} &=\{(0,0,0),(0,0,1),(0,0,2),(0,2,0),(0,2,1),(0,2,2),(0,4,0),(0,4,1),(0,4,2),(2,0,0),\\
&\quad\quad(2,0,1),(2,0,2),(2,2,0),(2,2,1),(2,2,2),(2,4,0),(2,4,1),(2,4,2)\} \\
X_{(0,1,0)}&=\{(0,1,0),(0,1,1),(0,1,2),(0,3,0),(0,3,1),(0,3,2),(0,5,0),(0,5,1),(0,5,2),(2,1,0),\\
&\quad\quad(2,1,1),(2,1,2),(2,3,0),(2,3,1),(2,3,2),(2,5,0),(2,5,1),(2,5,2)\} \\
X_{(1,0,0)}&=\{(1,0,0),(1,0,1),(1,0,2),(1,2,0),(1,2,1),(1,2,2),(1,4,0),(1,4,1),(1,4,2),(3,0,0),\\
&\quad\quad(3,0,1),(3,0,2),(3,2,0),(3,2,1),(3,2,2),(3,4,0),(3,4,1),(3,4,2)\} \\
X_{(1,1,0)}&=\{(1,1,0),(1,1,1),(1,1,2),(1,3,0),(1,3,1),(1,3,2),(1,5,0),(1,5,1),(1,5,2),(3,1,0),\\
&\quad\quad(3,1,1),(3,1,2),(3,3,0),(3,3,1),(3,3,2),(3,5,0),(3,5,1),(3,5,2)\} \\
\end{align*}
\end{ex}

With above notations, the orbit which is direct product of all even orbits and of transitive orbits is denoted by $X_{(0.\hdots,0)}$. We identify subscript as binary number and convert it into decimal number and from now on, we denote an arbitrary orbit with $X_i$. In particular we denote orbit $X_{(0.\hdots,0)}$ with symbol $X_0$. We denote the restrictions of maps $R_j$ to orbit $X_0$ with $T_j = R_j|_{X_0}$.

Now we discuss the structure of inner automorphism groups  generated by the set of restrictions of right multiplication maps to orbits.
\begin{lemma}\label{Special Case}
Let $(T,\triangleright)$ be Takasaki quandle for abelian group $G = \underbrace{\mathbb{Z}_{2} \times \hdots \mathbb{Z}_{2}}_\text{k~\text{ times }} \times \underbrace{\mathbb{Z}_{4} \times \hdots \times \mathbb{Z}_{4}}_\text{k'~times}$, then $Inn(X_0) \cong   \underbrace{\mathbb{Z}_{2} \times \hdots \mathbb{Z}_{2}}_\textbf{k'~times}$ .
\end{lemma}

\begin{proof}
We observe that the map $R_{e_{0}}|_{X_{0}}$ becomes identity map. Rest of the proof follows immediately from Lemma \ref{InnerAutomorphismGroup}.
\end{proof}

\begin{lemma}\label{Even Orbit}
Let $(T,\triangleright)$ be the Takasaki quandle for abelian group $ G = \mathbb{Z}_{n_{1}} \times \mathbb{Z}_{n_{2}} \times \hdots \times \mathbb{Z}_{n_{k}} \times \mathbb{Z}_{n_{k+1}} \times \hdots \times \mathbb{Z}_{n_{k+d}}$ with $gcd(n_i,2) = 2, ~\forall~ i=1,2,\hdots k$ and $gcd(n_i,2) = 1, ~\forall~i= k+1,k+2,\hdots ,k+d$ excluding the case of Lemma \ref{Special Case}. Then $Inn(X_0) \cong Inn(T)$.
\end{lemma}

\begin{proof}
We observe that proof of Lemma \ref{InnerAutomorphismGroup} holds as long as order of $R_{j}|_{X_{0}} = 2 $ for all $j \in T$ which holds true except the case discussed in Lemma \ref{Special Case}. Hence the result.
\end{proof}

\begin{lemma}\label{ShiftingThm}
Let $(T,\triangleright)$ be a Takasaki quandle and $X_{i}$ be an orbit of $T$. Then $Inn(X_{i}) \cong Inn(X_0)$.
\end{lemma}
\begin{proof}
First, for dihedral quandle $(\mathcal{R}_{n}, \triangleright)$, we observe that if $(i,2j-i)$ is a transposition in the map $R_{j} \in Inn(\mathcal{R}_{n})$, then $(i+1,2j-i+1)$ is a transposition in the map $R_{j+1}$. Moreover if  $i$ is a fixed point of the map $R_{j}$, then $i+1$ is a fixed point of the map $R_{j+1}$. 
Now since Takasaki quandle is direct product of dihedral quandles, so this shifting happens in each co-ordinate of right multiplication maps of Takasaki quandles. So any arbitrary orbit $X_i$ and the restricted maps $R_{j}|_{X_i}$  with $R_{j} \in Inn(T)$ can be identified as orbit $X_0$ and maps $R_{j}|_{X_{0}}$ respectively with changes in notation. Thus $Inn(X_{i}) \cong Inn(X_0)$.
\end{proof}

Now for the decomposition of $K[T]$, we require the understanding of conjugacy classes and representation theory of generalised dihedral groups. We recall them in the coming sections.

\section{Conjugacy Classes of Generalised Dihedral Group} \label{Cl}
Let $G$ be a group and $x \in G$. Then the centralizer of $x$ is defined as 
$$C_{G}(x) :=\{g \in G~|~gxg^{-1}=x\}$$

We know that $C_{G}(x)$ is subgroup of $G$. The size of conjugacy class of $x$ is given as 
$$|Cl(x)| = \frac{|G|}{|C_{G}(x)|}$$

Recall that generalised dihedral group $H \rtimes \mathbb{Z}_{2}$ where $H$ is a finite abelian group, has following finite presentation
$$Dih(H) = \{H,s~|~s^2=1,shs^{-1} = h^{-1},~\forall ~ h \in H\}$$
We assume that $Dih(H)$ is non abelian group i.e. $H$ is not isomorphic to elementary abelian $2-$group. Now we compute the center of group $Dih(H)$.
\begin{proposition}\label{Center}
Let $Dih(H)$ be generalised dihedral group $H \rtimes \mathbb{Z}_{2}$ where  $ H =  \mathbb{Z}_{n_{1}} \times \mathbb{Z}_{n_{2}} \times \hdots \times \mathbb{Z}_{n_{k}} \times \mathbb{Z}_{m_{1}} \times \hdots \times \mathbb{Z}_{m_{d}}$ where $gcd(n_{i},2) = 2 $,  $ \forall~1\leq i \leq k$, $gcd(m_{i},2) = 1$, $ \forall~ 1\leq i \leq d$. Let $\mathbb{Z}_{n_{i}} =  < r_{i} > $, for $1 \leq i \leq k $ and $\mathbb{Z}_{m_{i}} = <{r'}_{i}> $, for $1 \leq i \leq d $. Then center $Z(Dih(H))$ is given by
$$Z(Dih(H)) = \{{r_{1}}^{i_{1}}{r_{2}}^{i_2}\hdots{r_{k}}^{i_{k}}~~|~~i_{j} = n_{j}/2~ \text{ or }~ i_{j} = 0,~ 1 \leq j \leq k\}$$
\end{proposition}

\begin{proof}
We identify all the elements of $Dih(H)$ whose conjugacy class is singleton. From presentation of $Dih(H)$, we get that all elements of $Dih(H)$ are $\{h,sh~|~h \in H\}$. We have $sh \notin Z(Dih(H))$ as with $h'^{2} \neq 1$, we get $sh'shsh' = sh'^{2}h \neq sh$, implying $|Cl(sh)| >1$. We also have
$$|Cl(h)| = |Dih(H):C_{Dih(H)}(h)| \leq |Dih(H): H| \leq 2$$
Observe that  $shs^{-1} = h^{-1}$ and $shsh = 1$, so $g^{-1} \in Cl(g)$ for all $g \in Dih(H)$. This implies $Cl(h) = \{h,h^{-1}\}$ for all $h \in H$. % Thus if $g \in Z(Dih(H))$ then $g^2=1$.  Thus singleton conjugacy classes of $Dih(H)$ consist of $\{1\}$ and of order two elements of $H$.
Thus $|Cl(h)| = 1 $ if and only if $h = h^{-1}$. This proves our result.
\end{proof}

\begin{proposition}\label{Centralizer}
Let $Dih(H)$ be generalised dihedral group with above notations. Then for $h \in H$, we have 
$$C_{Dih(H)}(sh) =~ <sh,Z(Dih(H))>$$
\end{proposition}
\begin{proof}
Let $h' \in H$ and $h' \in C_{Dih(H)}(sh)$. Then $h'sh = shh'$, which gives us ${h'}^{-1} = h'$. Now Proposition \ref{Center} implies $h' \in Z(Dih(H)$. This proves $Z(Dih(H) \subset C_{Dih(H)}(sh)$.

Let $g = sh' \in C_{Dih(H)}(sh)$ with $h' \in H$. Then we get ${h'}^{-1}h = h^{-1}h' $. From Proposition \ref{Center}, this implies $h^{-1}h' \in Z(Dih(H))$. Thus we get $g = shk$ for some $k \in Z(Dih(H))$. This completes the proof.
\end{proof}

As a consequence of Proposition \ref{Centralizer}, we get
$$|C_{Dih(H)}(sh)|=2^{k+1}$$
and so 
$$|Cl(sh)| = \frac{|Dih(H)|}{2^{k+1}}$$

Thus $Dih(H)$ has $2^k$ singleton conjugacy classes, $|(|H|-2^k)/2|$ conjugacy classes with two elements and $2^{k}$ conjugacy classes with $|Dih(H)|/2^{k+1}$ elements.

\section{Irreducible representations of Generalised Dihedral Group} \label{reps}

In this section we discuss irreducible representations of generalised dihedral group $Dih(H)$. Since $Dih(H) = H \rtimes \mathbb{Z}_{2}$, so we induce irreducible representations of $H$ to $Dih(H)$. We use Clifford theory to check the irreducibility of induced representations. For more details we refer to \cite{SST}.

\begin{definition}{Induced representation}: Let $K$ be subgroup of $G$ and $(\phi,V)$ be a representation of $K$. The representation $\left( \rho,W \right)$ of $G$ induced from $\left( \phi , V \right)$ is given by :
$$ W = \bigoplus_{i=1}^{n}g_{i}V $$
where the set $\{g_1,g_2,\hdots g_n\}$ is the complete set of coset representatives of $K$ in $G$. We write $\rho = Ind_{K}^{G}\phi$ and
$$Dim~Ind_{K}^{G}{\phi} = (Dim~V) \cdot |G/K|$$ 
\end{definition}

The action of $G$ on $W$ is defined as below, it is enough to define this action on each direct summand in $W$. Let $g \in G$, $h \in K$ and $v \in V$ then $g.g_i = g_{\sigma(i)} h$ for some permutation $\sigma \in S_{n}$ and the action is
$$g\cdot(g_{i}v) = g_{\sigma(i)}hv$$
 %Note that action of $h$ on $V$ will give us $d \times d$ matrix $\phi_{h}$ where $d = Dim~V$.

Now we discuss the irreducibility of induced representations. Let $\hat{K}$ denote the set of irreducible representations of $K$. Let $\sigma \in \hat{K}$ and $g \in G$. Then $g$-conjugate representation of $\sigma$ is the representation $\prescript{g}{}{\sigma} \in \hat{K}$ defined by
$$\prescript{g}{}{\sigma}(h) = \sigma(g^{-1}hg)$$
for all $h \in K$.
The subgroup
$$I_{G}(\sigma) = \{g \in G : \prescript{g}{}{\sigma} \sim \sigma\} $$
of $G$ is called the inertia group of $\sigma \in \hat{K}$. It is easy to observe that $I_{G}(\sigma)$ contains $K$.

If $K$ is a subgroup of $G$ of index two then we can define an  alternating representation of $G$ with respect to $K$ as one dimensional representation $\varepsilon : G \to \mathbb{C}^{*}$ defined as 
$$\varepsilon(g)= \begin{cases}
                  1 \quad \quad \text{if } g\in K\\
                  -1 \text \quad \quad {otherwise}
                 \end{cases}$$
                 
Similar to $\hat{K}$, let $\hat{G}$ denote the set of irreducible representations of group $G$. Define 
$$\hat{G}(\sigma) = \{\theta \in \hat{G}~:~\sigma \preceq Res_{K}^{G}\theta\} = \{\theta \in \hat{G}~:~\theta \preceq Ind_{K}^{G}\sigma\}$$
where $\sigma \in \hat{K}$. Now we have following theorem
\begin{theorem}\cite[ Theorem 3.1]{ SST}\label{Clifford1}
Let $K$ be a subgroup of $G$ of index two. Then
\begin{enumerate}[(1).]
\item
If $I_{G}(\sigma) = K $, then $\theta := Ind_{K}^{G}(\sigma) \in \hat{G}, \theta \otimes \varepsilon = \theta$. %and $Res_{H}^{G} \theta = \sigma \oplus \prescript{h}{}{\sigma}$, with $\sigma$ and $\prescript{h}{}{\sigma}$ not equivalent.
\item
If $I_{G}(\sigma) = G $, then, for $\theta \in \hat{G}(\sigma)$, we have $Ind_{K}^{G}(\sigma) = \theta \oplus (\theta \otimes \varepsilon)$ with $\theta \not\sim \theta \otimes \varepsilon  $. % and $Res_{H}^{G} \theta = Res_{H}^{G} (\theta \otimes \varepsilon) = \sigma $.
\end{enumerate}
\end{theorem}

Now we apply this theory to get induced representations of $Dih(H)$ from irreducible representations of $H$. We consider two cases for $H$.

Case(i). Let $H = \mathbb{Z}_{n_{1}} \times \mathbb{Z}_{n_{2}} \times \hdots \times \mathbb{Z}_{n_{k}} $ be abelian group with $gcd(n_i,2)=1$ and with $\mathbb{Z}_{n_i} = <r_i>$ for all $1\leq i \leq k$. Let $\phi_{{i}_{j}} : \mathbb{Z}_{n_{i}} \to \mathbb{C}^{*}$ be an irreducible representation of $\mathbb{Z}_{n_{i}}$ for all $0 \leq i_{j} \leq n_{j}-1$. An  irreducible representation of $H$ is given as follows
$$\phi_{ \left( i_1,i_2,\hdots,i_k \right)} :  \mathbb{Z}_{n_{1}} \times \mathbb{Z}_{n_{2}} \times \hdots \times \mathbb{Z}_{n_{k}} \to \mathbb{C}^* $$
given as 
$$\phi_{ \left( i_1,i_2,\hdots,i_k \right)}{\left(m_1,m_2,\hdots,m_k \right)} = \phi_{i_1}(m_1) \hdots \phi_{i_k}(m_k) $$
where $m_{{i}} \in \mathbb{Z}_{n_i}$ for all $1 \leq i \leq k$. 

More specifically, we can write
$$\phi_{ \left( i_1,i_2,\hdots,i_k \right)}{\left(m_1,m_2,\hdots,m_k \right)} = \prod_{j=1}^{k}{\omega}_{n_{j}}^{m_{j}{i_{j}}}$$
where  $0\leq i_{j} \leq n_{j}-1$, $1\leq j \leq k$ and  ${\omega}_{n_{j}} = e^{{{\bf{i}}2\pi}/n_{j}}$.
We observe that $H$ has only one irreducible representation namely trivial representation with real character value. All remaining irreducible representations of $H$ have complex valued characters.
We refer to these representations as real irreducible and complex irreducible representations of $H$ respectively.

We recall that $Dih(H) := \left\{ H,s ~|~s^2,shs^{-1} = h^{-1}, ~\forall ~h \in H \right\}$ is the finite presentation of generalised dihedral group. Let $\{1,s\}$ be the set of coset representatives of $H$ in $Dih(H)$. Let $\phi_{\left(i_1,i_2,\hdots,i_k \right)}$ denote the irreducible representation of $H$. Then by using theory of induction, we get induced representation $\rho_{\left(i_1,i_2,\hdots,i_k \right)}$ of $Dih(H)$ as below
$$ \rho_{\left(i_1,i_2,\hdots,i_k \right)}(r_{j})=\begin{pmatrix}
              {\omega_{n_{j}}}^{i_j} &  0\\
              0 &  {\omega_{n_{j}}}^{-i_j} 
             \end{pmatrix}       \hspace{2cm}  \rho_{\left(i_1,i_2,\hdots,i_k \right)}(s)=\begin{pmatrix}
              0 &  1\\
              1 &  0
             \end{pmatrix}  $$

where $0 \leq i_{j} \leq n_{j}-1$ and $1 \leq j \leq k$. Equivalently as real representations we have following

$$ \rho_{\left(i_1,i_2,\hdots,i_k \right)}(r_{j})=\begin{pmatrix}
              cos(2{\pi}i_{j}/n_{j}) &  -sin(2{\pi}i_{j}/n_{j})\\
              sin(2{\pi}i_{j}/n_{j}) &  cos(2{\pi}i_{j}/n_{j})
             \end{pmatrix}       \hspace{2cm}  \rho_{\left(i_1,i_2,\hdots,i_k \right)}(s)=\begin{pmatrix}
              1 &  0\\
              0 &  -1
             \end{pmatrix}  $$

where $0 \leq i_{j} \leq n_{j}-1$ and $1 \leq j \leq k$. 

Let $\overline{\phi}$ denote the complex conjugate representation of $\phi$ of group $H$ defined as $\overline{\phi}(h) = \overline{\phi(h)}$ where $h\in H$. Then we have following result for abelian groups. 
\begin{lemma}\label{Abelian} \cite[Ex. 4.2]{BS}
Let $H$ be an abelian group and $\phi$ be an irreducible representation of $H$. Then $\phi(h^{-1}) = \overline{\phi}(h)$ for all $h \in H$.
\end{lemma}

\begin{proposition}\label{For generalised}
Let $Dih(H) = H \rtimes \mathbb{Z}_{2}$ be generalised dihedral group. Then
\begin{enumerate}
\item
If $\phi$ is a real irreducible representation of $H$ then $I_{Dih(H)}(\phi) = Dih(H)$.

\item
If $\phi $ is complex irreducible representation of $H$ then $I_{Dih(H)}(\phi) = H$.

\end{enumerate}
\end{proposition}
\begin{proof}
Since $H$ is an abelian group, $\phi $ is one-dimensional representation. By definition we have $H \subset I_{Dih(H)}(\phi)$. It is enough to check whether $s \in I_{Dih(H)}(\phi)$. We get 
$$\prescript{s}{}{\phi}(h) = \phi(shs^{-1}) = \phi(h^{-1}) = \overline{\phi(h)}$$
where the last equality follows from Lemma \ref{Abelian}. If $\phi$ is real representation then $\overline{\phi(h)} = \phi(h)$. This implies $s \in I_{Dih(H)}(\phi)$ and so $I_{Dih(H)}(\phi) = Dih(H)$.

If $\phi$ is a complex irreducible representation then $\overline{\phi(h)} \neq \phi(h)$ and so $s \notin I_{Dih(H)}(\phi)$. This implies $I_{Dih(H)}(\phi) = H$.
\end{proof}
Let $\phi $ be trivial representation of $H$. As a consequence of Proposition \ref{For generalised} together with Theorem \ref{Clifford1}, we get that  $Ind_{H}^{Dih(H)}\phi$ splits into two degree one irreducible representations of $Dih(H)$. We call them $\rho_{triv}$ and $\rho_{sign}$~, where
$$\rho_{triv}(g) = 1 \quad \forall~g \in Dih(H)$$
and
$$\rho_{sign}(g)= \begin{cases}
                  1 \quad \quad \text{if } g\in H\\
                  -1 \quad \quad \text{if } g\in sH\\
                 \end{cases}$$

Now we are left with $|H|-1$ degree two irreducible representations of $Dih(H)$ induced from non trivial irreducible representations of $H$. However some of them are equivalent representations. We use the following theorem to identify equivalent representations among these induced representations of $Dih(H)$.

\begin{theorem}\cite[ Corollary 2.1]{SST}\label{Conjugate}
 Let $K$ be normal subgroup of $G$. Then $Ind_{K}^{G} \sigma $ is irreducible if and only if $I_{G}(\sigma) = K$. Moreover, if $\sigma, \sigma_1 \in \hat{K}$ and $I_{G}(\sigma) = I_{G}(\sigma_{1}) = K$ then $Ind_{K}^{G}(\sigma) \sim Ind_{K}^{G}(\sigma_1)$ if and only if $\sigma $ and $\sigma_1$ are conjugate (i.e. there exists $g \in G$ such that $\sigma_1 = \prescript{g}{}{\sigma}$). 
\end{theorem}
Now we give following proposition for $Dih(H)$.

\begin{proposition}\label{Conjugategeneralised}
Let $Dih(H)$ be generalised dihedral group. Let $\phi$ be complex irreducible representation of $H$. Then $\phi$ and $\overline{\phi}$ are conjugate representations.
\end{proposition}
\begin{proof}
We compute $\prescript{s}{}{\phi}(h) = \phi(shs^{-1}) = \phi(h^{-1}) = \overline{\phi(h)} = \overline{\phi}(h)$. Therefore $\phi$ and $\overline{\phi}$ are conjugate representations of $H$.
\end{proof}
As a consequence of Theorem \ref{Conjugate} and Proposition \ref{Conjugategeneralised}, we get that  generalised dihedral group $Dih(H)$ has $(|H|-1)/2$ inequivalent degree two irreducible representations.
We denote these inequivalent degree two representations of $Dih(H)$ with $\psi_{j}$ where $1\leq j \leq (|H|-1)/2$.

Case(ii). Let $H= \mathbb{Z}_{n_{1}} \times \hdots \mathbb{Z}_{n_{k}} \times \mathbb{Z}_{m_{1}} \times \hdots \mathbb{Z}_{m_{d}} $ be an abelian group where $gcd(n_i,2)=2$ for $1\leq i\leq k$ and $ gcd(m_j,2)=1$ for $1\leq j \leq d$.

In this case, each $\mathbb{Z}_{n_i}$ has two real irreducible representations for $1\leq i \leq k$. So group $H$ has $2^{k}$ real irreducible representations. Now by using same approach as in odd case, we get that the group $Dih(H)$ has $2^{k+1}$ degree one  inequivalent irreducible representations and $(|H|-2^{k})/2$ degree two inequivalent irreducible representations. Note that out of these $2^{k+1}$ degree one representations, there are $2^{k}$ representations which map $s$ to  $-1$ and remaining $2^{k}$ representations  map $s$ to $1$.

We also require following theorem from \cite{JL} for decomposition of $K[T]$ as direct sum of right simple ideals.

\begin{theorem} \cite[Th. 14.17]{JL}. \label{JL}
Let $\chi_1,\hdots,\chi_k$ be complete set of irreducible characters of $G$. If $\psi$ is any character of $G$, then
$$\psi = d_1\chi_1+\hdots+d_k\chi_k$$
for some non-negative integers $d_1,d_2,\hdots,d_k$. Moreover,
$$d_i=~<\psi,\chi_i>~\text{ for } 1\leq i \leq k, \text{and}$$
$$<\psi,\psi>~ = \sum_{i=1}^{k}{d_{i}}^{2}$$

\end{theorem}

\section{Decomposition of Quandle rings of Takasaki Quandles} \label{s6}
In this section, we discuss decomposition of quandle rings of Takasaki quandles. We divide this section into three parts. We begin with Takasaki quandle of abelian group $H$ of odd order:
\subsection{When $H$ is an abelian group of odd order}

Let $(T,\triangleright)$ be Takasaki quandle for abelian group $H = \mathbb{Z}_{n_{1}} \times \mathbb{Z}_{n_{2}} \times \hdots \times \mathbb{Z}_{n_{k}}$ with $gcd(n_i,2) = 1, ~\forall~ i=1,2,\hdots k$. We observe that $T$ itself is an orbit under the action of $Inn(T)$.

Let $K=\mathbb{R}$ or $\mathbb{C}.$ We identify $T$ with $\{v_0, v_1, v_2, \dots , v_{|T|-1}\} \subseteq K[T]$, then $T$ can be considered as basis of $K[T]$. Now clearly the module generated by  $v_{triv}= \sum_{i=0}^{|T|-1} v_{i}$ is simple module over $K[T]$. We denote this module with $V_{triv}.$
 
Recall $Inn(T)$ is the inner automorphism group generated by right multiplication maps $R_{j}~,j \in T$. Consider the representation $\rho_{T}: {\rm Inn}(T)\to {\rm GL}(K[T])$ defined by $\rho(R_j)(v_i)=v_{i\triangleright j}$, for all $R_j\in {\rm Inn}(T).$ Let $\chi_{T}$ be the character associated to $\rho_{T}$. We identify $Inn(T)$ with $Dih(H):=\{H,s~|~s^2=1,shs^{-1}=h^{-1};~h \in H\}$ where $ < R_{e_{i}}R_{e_{0}}~|~ 1 \leq i \leq k >$ is identified as abelian group $H$ generated by set $\{r_{i}~|~ 1 \leq i \leq k\}$  and  $R_{e_{0}}$ with $s$.

\begin{lemma}\label{CharacterOdd}
With the same notations as discussed above, we have
$$\chi_{T}(g)= \begin{cases}
                  |H| \quad \quad \text{ if } g=1\\
                  1 \quad \quad \text { if } g \notin H \\
                  0 \text \quad \quad \text{ if } g \in H
                 \end{cases}$$
\end{lemma}

\begin{proof}
 Let $r_{i}$ be the generator of component $\mathbb{Z}_{n_{i}}$ of $H$. The group $Dih(H)$ has $(|H|+3)/2$ conjugacy classes,namely
$$\{1\}, \quad \{r_{1}^{i_{1}}\hdots r_{k}^{i_{k}},r_{1}^{-i_{1}}\hdots r_{k}^{-i_{k}}~|~1\leq i_{j} \leq (n_{j}-1)/2;~j=1,2,\hdots,k\}; \quad \{ hs; h \in H\}$$
Using the proof of Lemma \ref{InnerAutomorphismGroup} and Lemma \ref{Remark1},
$$\chi_{T}(1) = |H|,\quad \chi_{T}(h) = 0,\quad \chi_{T}(hs) = 1$$
for all $h \in H$. 
\end{proof}

The set $\{ \psi_j ~|~  1\leq j \leq (|H|-1)/2 \}$ is a complete set of inequivalent irreducible representations of $Dih(H)$ with degree two and let $V_{j}$ denote the simple module corresponding to $\psi_{j}$.

\begin{lemma}\label{InnerProductOdd}
 Let $k= (|H|-1)/2$. Then for each integer $j$ such that $1\leq j\leq k,$ the representation $\psi_{j}$ appears in the decomposition of $\rho_{T}$ with multiplicity $1.$  
\end{lemma}
\begin{proof}
 Let $\chi_j$ denote the character associated to the irreducible representation $\psi_{j}$ of $Dih(H).$ We recall that $\chi_{T}$ is character associated to $\rho_{T}$. For $1\leq j\leq k,$ using Lemma \ref{CharacterOdd} and section \ref{reps}, we compute
  $$ \langle \chi_{T}, \chi_j \rangle = \frac{1}{2|H|}\sum_{g\in Dih(H)} \chi_{T}(g) \chi_j(g)
  = 1.$$
Now the result follows from Theorem \ref{JL}.
\end{proof}

\begin{theorem} \label{OddFinal}
 Let $K= \mathbb{R}$ or $\mathbb{C}.$ Let $T$ be the Takasaki quandle for abelian group  $H = \mathbb{Z}_{n_{1}} \times \mathbb{Z}_{n_{2}} \times \hdots \times \mathbb{Z}_{n_{k}}$ with $gcd(n_i,2) = 1 ~\forall~ i=1,2,\hdots k$.
Then
$$K[T]= V_{triv}  \oplus \bigoplus_{j=1}^{(|H|-1)/2} V_{j}.$$

\end{theorem}

\begin{proof}
 It follows from the Lemma \ref{InnerProductOdd}, that if $\psi_j : Dih(H) \to {\rm GL}(V_j)$ is an irreducible representation of degree two, then $V_j$ is a simple $K[T]$-module with multiplicity $1.$ We denote this module by $V_{j}.$

Since $H$ is an abelian group of odd order and hence $H \rtimes \mathbb{Z}_{2}$ has $(|H|-1)/2$ inequivalent irreducible representations $\psi_j$ of degree two. Therefore $K[T]$ has $(|H|-1)/2$ distinct simple $K[T]$-modules $V_{j}$ for $1\leq j\leq (|H|-1)/2.$ Let $V_{triv}$ be one dimensional $K[T]$-module generated by  $v_{triv}.$ By dimension count, we get 
$$K[T]= V_{triv} \oplus \bigoplus_{j=1}^{(|H|-1)/2} V_{j}.$$
 
\end{proof}
\subsection{When H is an abelian group of even order}
Let $(T,\triangleright)$ be Takasaki quandle for abelian group $ H = \mathbb{Z}_{2n_1}\times \hdots \times \mathbb{Z}_{2n_{k}} \times \mathbb{Z}_{2n'_{1}} \times \hdots \times \mathbb{Z}_{2n'_{k'}} \times \mathbb{Z}_{m'_{1}} \times \hdots \times \mathbb{Z}_{m'_{d'}}$, where $gcd(n_{i},2) = 2$, $\forall$ $1\leq i \leq k$ and $gcd(n'_{i},2) = 1$, $\forall$ $1\leq i \leq k'$ and $gcd(m'_{i},2) = 1$, $\forall$ $1\leq i \leq d'$ except the case $H = \underbrace{\mathbb{Z}_{2} \times \hdots \mathbb{Z}_{2}}_\text{k~\text{ times }} \times \underbrace{\mathbb{Z}_{4} \times \hdots \times \mathbb{Z}_{4}}_\text{k'~times}$.

Let $X_{0}$ represent the orbit obtained by direct product of even orbits of dihedral quandles formed from each cyclic component of abelian group $H$. Then from Lemma \ref{Even Orbit}, we have $ Inn(X_0) =  \left( \mathbb{Z}_{n_1}\times \hdots \times \mathbb{Z}_{n_{k}} \times \mathbb{Z}_{m_{1}} \times \hdots \times \mathbb{Z}_{m_{d}} \right) \rtimes \mathbb{Z}_{2}$ where $gcd(n_{i},2)=2$ for $1\leq i\leq k$ and $gcd(m_{i},2) = 1$ for $1\leq i \leq d$. where $d=d'+k'$. Let $\tilde{H} = 2H$. We observe that $Inn(X_{0}) \cong Dih(\tilde{H})$.

Recall $Inn(X_{0})$ is the inner automorphism group generated by right multiplication maps $T_{j}~,j \in T$ where $T_{j} = R_{j}|_{X_{0}}$. Now we decompose $K[X_{0}]$ into direct sum of right simple ideals. Consider the representation $\rho_{X_0} : Inn(X_0) \to GL(K[X_0])$ defined by $\rho(T_{j})(v_{i}) = v_{i \triangleright j}$ for all $T_{j} \in Inn(X_0)$. Let $\chi_{X_0}$ be the character associated to $\rho_{X_0}$.

Let $Z$ be the set of right multiplication maps $T_{(a_1,a_2,\hdots,a_k,b_1,b_2,\hdots,b_{k'},c_1,c_2,\hdots,c_{d'})}$ where the co-ordinates have following properties
$$= \begin{cases}
                 0 \leq a_{i} < n_{i}, gcd(a_i,2)=2 \text{ or } 0 \quad \quad \text{ for } 1\leq i\leq k \\
                 0 \leq b_{i} < n'_{i}, \quad \quad \text{ for } 1\leq i\leq k'\\
                  0\leq c_{i} < m'_{i},\quad \quad \text{ for } 1\leq i \leq d' \\
                  
                 \end{cases}$$

 We identify $Inn(X_{0})$ with $Dih(\tilde{H}):=\{\tilde{H},s~|~s^2=1,shs^{-1}=h^{-1};~h \in \tilde{H}\}$ where $ < T_{e_{i}}T_{e_{0}}~|~ 1 \leq i \leq k+d >$ is identified as abelian group $\tilde{H}$ generated by set $\{r_{i}~|~ 1 \leq i \leq k+d\}$  and  $T_{e_{0}}$ with $s$. We note that out of elements of group $Inn(X_0)$, only elements of set $Z$ have fixed points belonging to orbit $X_{0}$. The size of set $Z$ is $|T|/(2^{2k+k'}) = |Dih(\tilde{H})|/2^{k+1}$. We observe from section \ref{Cl} that group $Dih(\tilde{H})$ have some conjugacy classes of size $|Dih(\tilde{H})|/2^{k+1}$. 
\begin{lemma}\label{CharacterEven}
With the same notations as above, we have
$$\chi_{X_0}(g)= \begin{cases}
                  |X_0| \quad \quad \text{if } g=1\\
                  2^k \quad \quad \text { if } g \in Cl(s)\\
                  0 \text \quad \quad {otherwise}
                 \end{cases}$$
\end{lemma}

\begin{proof}
The map $T_{e_0}$ has fixed point as $\left(\varepsilon_{1},\hdots \varepsilon_{k}, 0,\hdots 0,0,\hdots,0 \right)$ where  $\varepsilon_{i} = 0$  or $n_i$ for $1\leq i \leq k$. Thus $T_{e_0}$ has   $2^k$ fixed points. Observe $T_{e_{0}} \in Z$. We know that character is constant on conjugacy class. Thus the set  $Z$ is the conjugacy class of the map $T_{e_0}$. The rest of the proof is similar to Lemma \ref{CharacterOdd}. 

\end{proof}

\begin{remark}\label{Same}
The structure of orbits of the Takasaki quandle $(T,\triangleright)$ along with Lemma \ref{ShiftingThm} tells us that for an arbitrary orbit $X_i$ of $T$, above Lemma also holds true with corresponding versions of set $Z$.

\end{remark}

We know that the group $Dih(\tilde{H})$ has $2^{k+1}$ inequivalent irreducible representations of degree one. Let us recall that out of these $2^{k+1}$ degree one representations there are $2^{k}$ representations $\Phi_{j}$ such that $\Phi_{j}(s) = -1$ and $2^{k}$ representations $\phi_{j}$ are such that $\phi_{j}(s) = 1$. Let  $\{ \phi_j ~|~  1\leq j \leq 2^k \}$ be the set of latter inequivalent irreducible representations of $Dih(\tilde{H})$ of degree $1$ and let $U_{j}$ denote the corresponding simple module.

\begin{lemma}\label{InnerProductEvenDegreeone}
 Let $l=2^k $. Then for each integer $j$ such that $1\leq j\leq l,$ the representation $\phi_{j}$ appears in the decomposition of $\rho_{X_0}$ with multiplicity $1.$  
\end{lemma}
\begin{proof}
 Let $\chi_j$ denote the character associated to the irreducible representation $\phi_{j}$ of $Dih(\tilde{H}).$ We recall that $\chi_{X_0}$ is character associated to $\rho_{X_0}$. For $1\leq j\leq l,$ using Lemma \ref{CharacterEven} and section \ref{reps}, we compute
  $$ \langle \chi_{X_0}, \chi_j \rangle = \frac{1}{2|\tilde{H}|}\sum_{g\in Dih(\tilde{H})} \chi_{X_0}(g) \chi_j(g)
  = 1.$$
Now the result follows from Theorem \ref{JL}.
\end{proof}

Now let the set $\{ \psi_j ~|~  1\leq j \leq (|\tilde{H}|-2^{k})/2 \}$ is a complete set of inequivalent irreducible representations of $Dih(\tilde{H})$ with degree two and let $V_{j}$ denote the corresponding simple module.

\begin{lemma}\label{InnerProductEvenDegreetwo}
 Let $k= (|\tilde{H}|-2^k)/2$. Then for each integer $j$ such that $1\leq j\leq k,$ the representation $\psi_{j}$ appears in the decomposition of $\rho_{X_0}$ with multiplicity $1.$  
\end{lemma}
\begin{proof}
 Let $\chi'_j$ denote the character associated to the irreducible representation $\psi_{j}$ of $Dih(\tilde{H}).$ We recall that $\chi_{X_0}$ is character associated to $\rho_{X_0}$. For $1\leq j\leq k,$ using Lemma \ref{CharacterEven} and  section \ref{reps}, we compute
  $$ \langle \chi_{X_0}, \chi'_j \rangle = \frac{1}{2|\tilde{H}|}\sum_{g\in Dih(\tilde{H})} \chi_{X_0}(g) \chi'_j(g)
  = 1.$$
Now the result follows from Theorem \ref{JL}.
\end{proof}

\begin{lemma}\label{Evenorbit}
 Let $K= \mathbb{R}$ or $\mathbb{C}.$ Let $T$ denote the Takasaki quandle of even order defined as above and $X_0$ be it's even orbit. Then
 $$K[X_0]= \bigoplus_{j=1}^{2^k} U_{j} \oplus  \bigoplus_{j=1}^{(|\tilde{H}|-2^k)/2} V_{j}$$
\end{lemma}

\begin{proof}
 It follows from the Lemma \ref{InnerProductEvenDegreeone}, that if $\phi_j : Dih(\tilde{H}) \to {\rm GL}(U_j)$ is an irreducible representation of degree one, satisfying conditions of Lemma \ref{InnerProductEvenDegreeone}, then $U_j$ is a simple $K[X_0]$-module with multiplicity $1.$ %We denote this module by $U_{j}.$
 
Similarly from the Lemma \ref{InnerProductEvenDegreetwo}, it follows that if $\psi_j : Dih(\tilde{H}) \to {\rm GL}(V_j)$ is an irreducible representation of degree $2,$ then $V_j$ is a simple $K[X_0]$-module with multiplicity $1.$% We denote this module by $V_{j}.$
 
Now by dimension count, we get 
 $$K[X_0]= \bigoplus_{j=1}^{2^k} U_{j} \oplus  \bigoplus_{j=1}^{(|\tilde{H}|-2^k)/2} V_{j}$$
 
\end{proof}

From remark \ref{Same}, it follows that the decomposition of $K[X_i]$ is similar to  decomposition of $K[X_0]$ where $X_i$ is any arbitrary orbit of $T$. Let $U_{i,j}$ denote the simple $K[X_i]$-module corresponding to representation $\phi_{j}$ for $1 \leq j \leq 2^k$ and $V_{i,j}$ denote the simple $K[X_{i}]$-module corresponding to representation $\psi_{j}$ for $1 \leq j \leq (|\tilde{H}|-2^k)/2$. We have following theorem

\begin{theorem}
Let $K = \mathbb{R}$ or $\mathbb{C}$. Let $(T,\triangleright)$ be Takasaki quandle for abelian group $ H = \mathbb{Z}_{2n_1}\times \hdots \times \mathbb{Z}_{2n_{k}} \times \mathbb{Z}_{2n'_{1}} \times \hdots \times \mathbb{Z}_{2n'_{k'}} \times \mathbb{Z}_{m'_{1}} \times \hdots \times \mathbb{Z}_{m'_{d'}}$. where $gcd(n_{i},2) = 2$ for $1\leq i \leq k$ and $gcd(n'_{i},2) = 1$ for $1\leq i \leq k'$ and $gcd(m'_{i},2) = 1$ for $1\leq i \leq d'$. Then
$$K[T]= \bigoplus_{i=1}^{2^{k+k'}} \left( \bigoplus_{j=1}^{2^k} U_{i,j} \oplus  \bigoplus_{j=1}^{(|\tilde{H}|-2^k)/2} V_{i,j}  \right)$$
\end{theorem}

\begin{proof}
The proof of this theorem follows from remark \ref{Same} and Lemma \ref{Evenorbit}.
\end{proof}
\subsection{Special case}
Now we deal with quandle ring $K[T]$ where $K = \mathbb{R}$ or $\mathbb{C}$ and $T$ is Takasaki quandle of even order for group $H = \underbrace{\mathbb{Z}_{2} \times \hdots \mathbb{Z}_{2}}_\text{k~\text{ times }} \times \underbrace{\mathbb{Z}_{4} \times \hdots \times \mathbb{Z}_{4}}_\text{k'~times}$ .

The Takasaki quandle defined above has $2^{k+k'}$ orbits of size $2^{k'}$. Recall $X_{0}$ represents the orbit obtained by direct product of even orbits of dihedral quandles formed from each cyclic component of abelian group $H$. Let ${\rho}_{X_0} : Inn(X_0) \to K[X_0]$ be the representation of $Inn(X_0).$ Let $\chi_{X_0}$ be the character associated to $\rho_{X_0}$.

\begin{lemma}\label{CharacterSpecialcase}
With the same notations as above, we have
$$\chi_{X_0}(g)= \begin{cases}
                  |X_0| \quad \quad \text{if } g=1\\
                  0 \text \quad \quad {otherwise}
                 \end{cases}$$
\end{lemma}
\begin{proof}
Since $T_{j} = R_{j}|_{X_{0}}$. We see that $T_{e_{0}} = 1$ and $T_{e_{i}}T_{e_{0}} = 1$ if $i'th$ component of $H$ is $\mathbb{Z}_{2}$. Otherwise $T_{e_{i}}T_{e_{0}}$ is element of order two with no fixed points. Similarly we can say that every non identity element is of order two with no fixed points. Thus the result follows.
\end{proof}
In this case $Inn(X_0)$ is abelian group $\underbrace{\mathbb{Z}_{2} \times \hdots \mathbb{Z}_{2}}_\text{k'~times}$. So $Inn(X_0)$ has $2^{k'}$ irreducible inequivalent representations of degree one. Let the set $\{ \psi_j ~|~  1\leq j \leq 2^{k'} \}$ be the complete set of inequivalent irreducible representations of $Inn(X_0)$ with degree one. Note that $|X_{0}| = |Inn(X_{0}|$.

\begin{lemma}\label{InnerProductSpecialcase}
 Let $l= 2^{k'}$. Then for each integer $j$ such that $1\leq j\leq l,$ the representation $\psi_{j}$ appears in the decomposition of $\rho_{X_0}$ with multiplicity $1.$  
\end{lemma} 

\begin{proof}
 Let $\chi_j$ denote the character associated to the irreducible representation $\psi_{j}$ of $Inn(X_0).$ We recall that $\chi_{X_0}$ is character associated to $\rho_{X_0}$. For $1\leq j\leq l,$ using Lemma \ref{CharacterSpecialcase} and section \ref{reps}, we compute
  $$ \langle \chi_{X_0}, \chi_j \rangle = \frac{1}{2^{k'}}\sum_{g\in {\rm Inn}(X_0)} \chi_{X_0}(g) \chi_j(g)
  = 1.$$
The result now follows from Theorem \ref{JL}.
\end{proof}

 Now we have following theorem

\begin{theorem}
 Let $K= \mathbb{R}$ or $\mathbb{C}.$ Let $T$ be the Takasaki quandle for abelian group $$H = \underbrace{\mathbb{Z}_{2} \times \hdots \mathbb{Z}_{2}}_\text{k~\text{ times }} \times \underbrace{\mathbb{Z}_{4} \times \hdots \times \mathbb{Z}_{4}}_\text{k'~times}$$
Then $$K[T] = \bigoplus_{i=1}^{2^{k+k'}} \left( \bigoplus_{j=1}^{2^{k'}} U_{i,j}  \right).$$
\end{theorem}

\begin{proof}
Let $X_i$ be any arbitrary orbit of $T$. Let $U_{i,j}$ denote the simple $K[X_i]$-module corresponding to representation $\psi_{j}$ of $Inn(X_{i})$ for $1 \leq j \leq 2^{k'}$. For $i=0$, using Lemma \ref{InnerProductSpecialcase}, we get 
$$K[X_{0}] = \bigoplus_{j=1}^{2^{k'}}U_{0,j}$$
Following the arguments as that of even case above, the decomposition of $K[X_i]$ is similar to decomposition of $K[X_0]$. Since $T$ has $2^{k+k'}$ orbits, we get

 $$K[T] = \bigoplus_{i=1}^{2^{k+k'}} \left( \bigoplus_{j=1}^{2^{k'}} U_{i,j}  \right).$$
 
\end{proof}

\bibliographystyle{amsalpha}

\end{document}